\documentclass[reqno]{amsart}
\usepackage{float}
\usepackage{amsmath}
\usepackage{graphicx}
\usepackage{latexsym}
\usepackage{amsfonts}
\usepackage{amssymb}
\theoremstyle{plain}
\newtheorem{theorem}{Theorem}
\newtheorem{corollary}[theorem]{Corollary}
\newtheorem{lemma}[theorem]{Lemma}

\theoremstyle{definition}

\newtheorem*{remark*}{Remark}

\begin{document}
\title[Exit times for integrated random walks]{Exit times for integrated random walks}
\author[Denisov]{Denis Denisov}
\address{School of Mathematics, Cardiff University, Senghennydd Road
CARDIFF, Wales, UK.
CF24 4AG Cardiff }
\email{dennissov@gmail.com}

\author[Wachtel]{Vitali Wachtel}
\address{Mathematical Institute, University of Munich, Theresienstrasse 39, D--80333
Munich, Germany}
\email{wachtel@mathematik.uni-muenchen.de}

\begin{abstract}
We consider a centered random walk with finite variance and investigate the asymptotic behaviour
of the probability that the area under this walk remains positive up to a large time $n$. 
Assuming that the moment of order $2+\delta$ is finite, we show that the exact asymptotics for this probability are 
$n^{-1/4}$. To show these asymptotics we develop a discrete potential theory for the integrated random walk. 
\end{abstract}


\keywords{Markov chain, exit time, harmonic function, normal approximation, Kolmogorov diffusion}
\subjclass{Primary 60G50; Secondary 60G40, 60F17}
\thanks{Supported by the DFG}
\maketitle
\section{Introduction, main results and discussion}
\subsection{Background and motivation}
Let $X,X_1,X_2,\ldots$ be independent identically distributed random variables with $\mathbf{E}[X]=0$.
For every starting point $(x,y)$ define
$$
S_n=y+X_1+X_2+\ldots+X_n, \quad n\geq0.
$$ 
and
$$
S_n^{(2)}=x+S_1+S_2+\ldots+S_n=x+ny+nX_1+(n-1)X_2+\ldots+X_n.
$$
Sinai \cite{Sin92} initiated the study of the asymptotics of the probability of the event
$$
A_n:=\left\{S_k^{(2)}>0\text{ for all }k\leq n\Big|S_0=S_0^{(2)}=0\right\}.
$$
Assuming that $S_n$ is a simple symmetric random walk he showed that 
\begin{equation}\label{sinai}
C_1n^{-1/4}\leq\mathbf{P}(A_n)\leq C_2n^{-1/4}.
\end{equation}
The same bounds were obtained for some other special cases in \cite{Vys08}. 

Aurzada and Dereich \cite{AD11} have shown that if $\mathbf{E}e^{\beta |X|}<\infty$ for some
positive $\beta$ then
\begin{equation}
\label{AD}
C_*n^{-1/4}\log^{-\gamma}n\leq\mathbf{P}(A_n)\leq
C^*n^{-1/4}\log^{\gamma}n
\end{equation}
with some positive constants $C_*$, $C^*$ and some finite $\gamma$.  Bounds (\ref{AD}) are just a 
special case of the results in \cite{AD11} for $q$-times integrated random walks and Levy
processes.
Dembo, Ding and Gao \cite{DG11} have recently shown that (\ref{sinai}) is valid for all random walks
with finite second moment. 

Exact asymptotics for $\mathbf{P}(A_n)$ are known only in some special cases.
Vysotsky \cite{Vys11} have shown that if, in addition to the second moment assumption,
$S_n$ is either right-continuous 
or right-exponential  then
\begin{equation}\label{vlad}
\mathbf{P}(A_n)\sim Cn^{-1/4}.
\end{equation}
(Here and throughout $a_n\sim b_n$ means that $\frac{a_n}{b_n}\to1$ as $n\to\infty$.)

It is natural to expect that (\ref{vlad}) holds for all driftless random walks with finite variance.

If one replaces the second moment condition by the assumption that $X$ belongs to the normal domain of
attraction of the spectrally positive $\alpha$-stable law with some $\alpha\in(1,2]$, then (\ref{sinai}) 
and (\ref{vlad}) remain valid with $n^{-(\alpha-1)/2\alpha}$ instead of $n^{-1/4}$, see \cite{DG11} and \cite{Vys11}.

The methods used in the above mentioned papers are quite different. It is not clear what is the most natural tool 
for this problem. Here we propose another approach to this problem.  
More precisely, we develop a potential theory for integrated random walks, which allows one
to determine the exact asymptotic behaviour of $\mathbf{P}(A_n)$. It can be seen as a continuation of our studies of 
exit times of multi-dimensional random walks, see \cite{DW10,DW11}.

It is clear that the sequence $\{S_n^{(2)}\}_{n\geq1}$ is non-markovian. 
This fact complicates the analysis  of the integrated random walk. However, it
is possible to obtain the markovian property by increasing the dimension of the process.
More precisely, we consider the process
$$
Z_n:=(S_n^{(2)},S_n).
$$
Then, the first time when $S_n^{(2)}$ is not positive coincides with the exit time of $Z_n$ from a half-space
$$
\tau:=\min\{k\geq1:Z_k\notin\mathbb{R}_+\times\mathbb{R}\}.
$$

In our recent paper \cite{DW11} we suggested a method of studying random walks conditioned to stay in a cone. 
Similarly in the case of the integrated random walks we have a (quite simple) cone $\mathbb{R}_+\times\mathbb{R}$,
but the process $Z_n$ is 'really' Markov, i.e. the increments are not independent. We show that the method from 
\cite{DW11} can be adapted to the case of Markov chain $Z_n$, and this adaptation allows one to find asymptotics
of $\mathbf{P}_z(\tau>n)$ for every starting point $z=(x,y)$.
\subsection{Main result}
Our approach essentially relies on a strong normal approximation and corresponding results 
for the integrated Brownian motion. Hence we will start with results and notation for the integrated Brownian motion. 
This process is also known as the Kolmogorov diffusion. 

Let $B_t$ be a standard Brownian motion and consider a two-dimensional process $(\int_0^tB_sds,B_t)$. Since this
process is gaussian, one can  obtain by computing correlations that the transition density of $(\int_0^tB_sds,B_t)$ is given by
$$
g_t(x,y;u,v)=\frac{\sqrt{3}}{\pi t^2}\exp\left\{-\frac{6(u-x-ty)^2}{t^3}+\frac{6(u-x-ty)(v-y)}{t^2}
-\frac{2(v-y)^2}{t}\right\}.
$$

Let
$$
\tau^{bm}:=\min\left\{t>0:x+yt+\int_0^t B_sds\leq0\right\}.
$$
The behaviour of the killed at leaving $\mathbb{R}_+\times\mathbb{R}$ version of 
$\left(\int_0^t B_s ds,B_t\right)$ was by many authors. Here we will follow a paper by Groeneboom, Jongbloed and
Wellner \cite{GJW99}, where  one can also find a  history of the subject  and corresponding references. 
In particular they found the harmonic function for this 
process, which is given by the following relations:
\begin{equation}\label{h-def}
h(x,y)=
\begin{cases}
\left(\frac{2}{9}\right)^{1/6}\frac{y}{x^{1/6}}U\left(\frac{1}{6},\frac{4}{3},\frac{2y^3}{9x}\right),\quad y\ge 0 \\
-\left(\frac{2}{9}\right)^{1/6}\frac{1}{6}\frac{y}{x^{1/6}}e^{2y^3/9x}U\left(\frac{7}{6},\frac{4}{3},-\frac{2y^3}{9x}\right),\quad y<0,
\end{cases}
\end{equation}
where $U$ is the confluent hypergeometric function. Function $h(x,y)$ is harmonic in the sense that $\mathcal{D}h=0$, 
where $\mathcal{D}=y\frac{\partial}{\partial x}+\frac{1}{2}\frac{\partial^2}{\partial y^2}$ is the generator of
$\left(\int_0^t B_s ds,B_t\right)$.
Using the explicit density of $\mathbf P_{(0,1)}(\tau^{bm}>t)$ found in \cite{McK63},
they derived asymptotics for 
\begin{equation}
\label{asym.bm}
 \mathbf P_{(x,y)}(\tau^{bm}>t)\sim \varkappa\frac{h(x,y)}{t^{1/4}},\quad t\to \infty,
\end{equation}
where $\varkappa=\frac{3\Gamma(1/4)}{2^{3/4}\pi^{3/2}}.$

The harmonic function $h$ defined in (\ref{h-def}) helps us to construct the corresponding 
harmonic function for the killed integrated random walk. As function $h$ is 
defined only for $z\in\mathbb{R}_+\times\mathbb{R}$,  we extend it to 
$\mathbb{R}^2$ by putting $h=0$ outside $\mathbb{R}_+\times\mathbb{R}$. Function $h$ is harmonic for the killed integrated Brownian motion but not  
for the killed integrated random walk. To overcome this difficulty we introduce a corrector function  for $z=(x,y)\in\mathbb{R}^2$,
\begin{equation}
  \label{eq:defn.f}
  f(x,y)=\mathbf E_{z}h(Z(1))-h(z). 
\end{equation}
This function is well defined since we have extended $h$ to the whole plane. Now we are in position 
to define the harmonic function for the killed integrated random walk. For $z\in \mathbb{R}_+\times\mathbb{R}$ 
let  
\begin{equation}
  \label{eq:V}
  V(z)=h(z)+\mathbf E_z\sum_{l=1}^{\tau-1}f(Z_k).
\end{equation}
This function is harmonic for the killed integrated random walk in the sense that 
\begin{equation}
\label{T1.1}
\mathbf{E}_z\left[V(Z_1),\tau>1\right]=V(z)\text{ for all }z\in \mathbb{R}_+\times\mathbb{R}.
\end{equation}
It is not at all clear that function $V$ in (\ref{eq:V}) 
is well-defined and positive. In fact this is the most difficult part of the proof and it is done 
in Section~\ref{sect.V.is.good}. 

Our main result is the following theorem.
\begin{theorem}
\label{main}
Assume that $\mathbf{E}X=0$, $\mathbf{E}[X^2]=1$ and $\mathbf{E}|X|^{2+\delta}<\infty$ for some $\delta>0.$
Then the function $V$ from (\ref{eq:V}) is well-defined 
and  strictly positive on
$$
K_+:=\left\{z:\,\mathbf{P}_z(Z_n\in\mathbb{R}_+\times\mathbb{R}_+,\tau>n)>0\text{ for some }n\geq0\right\}.
$$
Moreover,
\begin{equation}
\label{T1.2}
\mathbf{P}_z(\tau>n)\sim \varkappa\frac{V(z)}{n^{1/4}}\quad\text{ as }n\to\infty
\end{equation}
and 
\begin{equation}
\label{T1.3}
\mathbf{P}_z\left(\left(\frac{S_n^{(2)}}{n^{3/2}},\frac{S_n}{n^{1/2}}\right)\in\cdot\bigg|\tau>n\right)
\to \mu\quad\text{weakly},
\end{equation}
where $\mu$ has density 
$$
Ch(x,y)g_1(0,0;x,y),\quad (x,y)\in\mathbb{R}_+\times\mathbb{R}.
$$
\end{theorem}
{F}rom (\ref{T1.2}) and the total probability formula we obtain
\begin{corollary}
For every random walk satisfying the conditions of Theorem \ref{main} holds
$$
\mathbf{P}_0(A_n)\sim\frac{C}{n^{1/4}}
$$
with
$$
C=\mathbf{E}[V((X,X)), X>0].
$$
\end{corollary}
\subsection{Local asymptotics for integrated random walks.}
Caravenna and Deuschel \cite{CD08} have proven a local limit theorem for $Z_n$ under the assumption
that the distribution of $X$ is absolutely continuous. Using similar arguments one can show that if
$X$ is $\mathbb{Z}$-valued and aperiodic then
\begin{equation}\label{Loc.1}
\sup_{\tilde{z}}\left| n^2\mathbf{P}_z(Z_n=\tilde{z})-
g_1\left(0,0;\frac{\tilde{x}}{n^{3/2}},\frac{\tilde{y}}{n^{1/2}}\right)\right|\to0.
\end{equation}
Combining this unconditioned local limit theorem with (\ref{T1.3}) one can derive a conditional local limit theorem:
\begin{equation}\label{Loc.2}
\sup_{\tilde{z}}\left| n^{2+1/4}\mathbf{P}_z(Z_n=\tilde{z},\tau>n)-
\varkappa V(z)h\left(\frac{\tilde{x}}{n^{3/2}},\frac{\tilde{y}}{n^{1/2}}\right)
g_1\left(0,0;\frac{\tilde{x}}{n^{3/2}},\frac{\tilde{y}}{n^{1/2}}\right)\right|\to0.
\end{equation}
Furthermore, for every fixed $\tilde{z}\in K_+$,
\begin{equation}
\label{Loc.3}
\lim_{n\to\infty}n^{2+1/2}\mathbf{P}_z(Z_n=\tilde{z},\tau>n)=V(z)V'(z)
\end{equation}
with some positive function $V'$.

The proof of (\ref{Loc.2}) and (\ref{Loc.3}) repeats virtually word by word the proof of local
asymptotics in \cite{DW11}, see Subsection 1.4 and Section 6 there. For this reason we 
do not give a proof of these statements.

Having (\ref{Loc.2}) one can easily show that
$$
\mathbf{P}_0(A_n|Z_{n+2}=0)\sim\frac{C}{n^{1/2}}
$$
with some positive constant $C$. A slightly weaker form of this relation was conjectured by 
Caravenna and Deuschel \cite[equation (1.22)]{CD08}.

 Aurzada,  Dereich and Lifshits \cite{ADL12} have recently obtained lower and upper bounds 
for the integrated simple random walk, 
$$
cn^{-1/2}\le \mathbf P_0(S^{(2)}_1\ge 0,\ldots,S^{(2)}_{4n}\ge 0 | S_{4n}=0,S^{(2)}_{4n}=0)\le Cn^{-1/2}. 
$$

\subsection{Organisation of the paper} In \cite{DW11} we have suggested a method
of investigating exit times from cones for random walks. In the present paper we have a Markov
chain instead of a random walk with independent increments. But it turns out that this fact
is not important, and the method from \cite{DW11} works also for Markov processes. 

The first step consists in construction of the harmonic function $V_0(z)$. As in \cite{DW11} we
start from the harmonic function for the corresponding limiting process. Obviously,
$$
\left(\frac{S_{[nt]}^{(2)}}{n^{3/2}},\frac{S_{[nt]}}{n^{1/2}}\right)\Rightarrow
\left(\int_0^t B_s ds,B_t\right).
$$

We then define for every $z\in\mathbb{R}_+\times\mathbb{R}$
\begin{equation}
\label{limit}
V_0(z)=\lim_{n\to\infty}\mathbf{E}_z[h(Z_n),\tau>n].
\end{equation}
The justification of this formal definition is the most technical part of our approach.
It is worth mentioning that we can not just repeat the proof from  \cite{DW11}. There
we used a certain a-priori information on the behaviour of first exit times. (It was some
moment inequalities, which were already known in the literature.) For integrated random walks 
we do not have such information and, therefore, should find an alternative way of justification
of (\ref{limit}). This is done in Section \ref{sect.V.is.good}.

Having constructed harmonic function for $Z_n$ we follow our approach in \cite{DW10,DW11} and apply 
the KMT-coupling to obtain the asymptotics for $\tau$. (This explains our moment condition 
in Theorem \ref{main}.) For details see Section \ref{Sec.tau}. We omit the proof of (\ref{T1.3}),
since it is again a repetition of the corresponding arguments in \cite{DW11}. 
For integrated random walks a strong approximation was used in Aurzada and Dereich \cite{AD11}
to obtain (\ref{AD}). This formula shows that a direct, without use of potential theory, 
application of coupling produces a superfluous logarithmic terms even under exponential moment assumption.

Finally we show that the definitions (\ref{eq:V}) and (\ref{limit}) are equivalent and $V(z)=V_0(z)$. 

\subsection{Conclusion}
In our previous works \cite{DW10,DW11} we showed that Brownian asymptotics for exit times 
can be transferred to  exit times for multidimensional random walks. 
In the present work we consider an integrated random walk which can be viewed as  a two-dimensional Markov chain.  
We study exit times from a half-space and transfer the corresponding results for the Kolmogorov diffusion. 
These examples make plausible the following hypothesis. 

Let $X_n$ be a Markov chain, $D$ be an unbounded domain and 
$\tau_D:=\min\{n\ge 1: X_n \notin D\}$. Assume that this Markov chain, properly scaled,  
converges as a process to a diffusion $Y_t,t\ge 0$.  Assume also that 
the exit time of this diffusion $T_D:=\min\{t\ge 0: Y_t \notin D\}$ has the following 
asymptotics
$$
\mathbf P_y(T_D>t)\sim \frac{h(y)}{t^p},\quad t\to\infty,
$$
where $h(y)$ is the corresponding harmonic function of the killed diffusion  $Y_{t\wedge T_D}$. Then, there 
exists a positive harmonic function $V(x)$ for the killed Markov chain ${X_{n\wedge \tau_D}}$ such that 
$$
\mathbf P_x(\tau_D>n)\sim \frac{V(x)}{n^p},\quad n\to\infty.
$$
Naturally, this general theorem will require some moment assumptions and some 
assumptions on the smoothness of the unbounded domain $D$. Since we have a convergence of processes the domain $D$ 
should have certain scaling properties. Hence it seems natural 
for the domain $D$ to be a cone, at least asymptotically.

\section{Construction of harmonic function}\label{sect.V.is.good}
This section is devoted to the construction of the harmonic function
$V_0$. Let
\begin{align}
  Y_0&=h(z),\nonumber\\
  \label{eq:defn.y}
  Y_{n+1}&=h(Z_{n+1})-\sum_{k=0}^n f(Z_k), \quad n\ge 0.
\end{align}
\begin{lemma}\label{lem.martingale}
  The sequence $Y_n$ defined in (\ref{eq:defn.y}) is a martingale. 
\end{lemma}
\begin{proof}
  Clearly,
  \begin{align*}
    \mathbf{E}_z\left[Y_{n+1}-Y_n|\mathcal{F}_n\right]
    &=\mathbf{E}_{z}\left[\left(h(Z_{n+1})-h(Z_n)-f(Z_n)\right)|\mathcal{F}_n\right]\\
    &=-f(Z(n))+\mathbf{E}_{z}\left[\left(h(Z_{n+1})-h(Z_n)\right)|Z_n\right]\\
    &=-f(Z_n)+f(Z_n)=0,
  \end{align*}
  where we used the definition of the function $f$ in
  \eqref{eq:defn.f}.
\end{proof}
Before proceeding any further we need to study some properties of the
functions $h(x,y)$ and $f(x,y)$. 
\begin{lemma}
  \label{derivatives}
  Function $h$ has the following partial derivatives, 
\begin{equation}
\label{partial.x}
\frac{\partial^i h(x,y)}{\partial x^i}=
\begin{cases}
C_i \left(\frac{2}{9}\right)^{1/6}\frac{y}{x^{1/6+i}}U\left(\frac{1}{6}+i,\frac{4}{3},\frac{2y^3}{9x}\right),\quad x\geq0, y\ge 0 \\
-\left(\frac{2}{9}\right)^{1/6}\frac{1}{6}\frac{y}{x^{1/6+i}}e^{2y^3/9x}U\left(\frac{7}{6}-i,\frac{4}{3},-\frac{2y^3}{9x}\right),\quad x\geq0,y<0,
\end{cases}
\end{equation}
for $i\ge 0$ and 
\begin{equation}
\label{partial.xy}
\frac{\partial^{i+1} h(x,y)}{\partial x^i\partial y}=
\begin{cases}
\frac{-3}{i-1/6}C_i \left(\frac{2}{9}\right)^{1/6}\frac{1}{x^{1/6+i}}U\left(\frac{1}{6}+i,\frac{1}{3},\frac{2y^3}{9x}\right),\quad x\geq0,y\ge 0 \\
\left(\frac{2}{9}\right)^{1/6}\frac{1}{2}\frac{1}{x^{1/6+i}}e^{2y^3/9x}U\left(\frac{7}{6}-i-1,\frac{1}{3},-\frac{2y^3}{9x}\right),\quad x\geq0,y<0.
\end{cases}
\end{equation}
Here, $C_0=1$ and $C_{i+1}=-C_i(i+1/6)(i-1/6)$ for $i\ge 0$.  
\end{lemma}
\begin{proof}
We will prove (\ref{partial.x}) by induction. The base of induction $i=0$ corresponds to the definition 
of $h$. Now suppose that (\ref{partial.x}) is true for $i$ and prove it for $i+1$.  

Consider first $y\ge 0$. By the induction hypothesis,  
\begin{align*}
  \frac{\partial^{i+1} h(x,y)}{\partial x^{i+1}}&=
C_i \left(\frac{2}{9}\right)^{1/6} 
\frac{\partial}{\partial x} \left[\frac{y}{x^{1/6+i}}U\left(\frac{1}{6}+i,\frac{4}{3},\frac{2y^3}{9x}\right)\right]\\
&= -C_i \left(\frac{2}{9}\right)^{1/6}\frac{y}{x^{1/6+i+1}}\biggl((1/6+i)U\left(\frac{1}{6}+i,\frac{4}{3},\frac{2y^3}{9x}\right)\\
&\hspace{1cm}+ \frac{2y^3}{9x}U'\left(\frac{1}{6}+i,\frac{4}{3},\frac{2y^3}{9x}\right)\biggr)\\
&=-C_i \left(\frac{2}{9}\right)^{1/6}\left(i+1/6\right)\left(i-1/6\right)\frac{y}{x^{1/6+i+1}}U\left(\frac{1}{6}+i+1,\frac{4}{3},\frac{2y^3}{9x}\right),
\end{align*}
where we applied (13.4.23) of \cite{AS64} in the last step. Recalling the definition of $C_{i+1}$ we see that 
(\ref{partial.x}) holds for $i+1$ and positive $y$. 

Consider second the case $y<0$.  By the induction hypothesis,
\begin{align*}
  \frac{\partial^{i+1} h(x,y)}{\partial x^{i+1}}&=
-\left(\frac{2}{9}\right)^{1/6}\frac{1}{6} \frac{\partial}{\partial x} \left[e^{2y^3/9x}\frac{y}{x^{1/6+i}}U\left(\frac{7}{6}-i,\frac{4}{3},-\frac{2y^3}{9x}\right)\right]\\
&= -\left(\frac{2}{9}\right)^{1/6}\frac{1}{6}\frac{y}{x^{1/6+i+1}}e^{2y^3/9x}
\biggl(-(1/6+i)U\left(\frac{7}{6}-i,\frac{4}{3},-\frac{2y^3}{9x}\right)\\
&\hspace{1cm}-\frac{2y^3}{9x}U\left(\frac{7}{6}-i,\frac{4}{3},-\frac{2y^3}{9x}\right)
+\frac{2y^3}{9x}U'\left(\frac{7}{6}-i,\frac{4}{3},-\frac{2y^3}{9x}\right)\biggr)\\
&=
-\left(\frac{2}{9}\right)^{1/6}\frac{1}{6}\frac{y}{x^{1/6+i+1}}e^{2y^3/9x}U\left(\frac{7}{6}-i-1,\frac{4}{3},-\frac{2y^3}{9x}\right),
\end{align*}
where we applied (13.4.26) of \cite{AS64} in the final step. 
This proves (\ref{partial.x}) for negative values of $y$. 

The proof of (\ref{partial.xy}) is similar. First we prove it for $y\ge 0$. 
Using (\ref{partial.x}), 
  \begin{align*}
    \frac{\partial^{i+1} h(x,y)}{\partial x^i\partial y}&=C_i \left(\frac{2}{9}\right)^{1/6}
  \frac{\partial}{\partial y }
     \frac{y}{x^{1/6+i}}U\left(\frac{1}{6}+i,\frac{4}{3},\frac{2y^3}{9x}\right)\\
  &=C_i \left(\frac{2}{9}\right)^{1/6}\frac{-3}{x^{1/6+i}}\biggl(
      -\frac{1}{3}U\left(\frac{1}{6}+i,\frac{4}{3},\frac{2y^3}{9x}\right)\\
  &\hspace{1cm}-\frac{2y^3}{9x}U'\left(\frac{1}{6}+i,\frac{4}{3},\frac{2y^3}{9x}\right)\biggr)\\
  &=C_i \left(\frac{2}{9}\right)^{1/6}\frac{-3}{(i-1/6)x^{1/6+i}}U\left(\frac{1}{6}+i,\frac{1}{3},\frac{2y^3}{9x}\right),
\end{align*}
where we used (13.4.24) of \cite{AS64} in the final step. Finally, for $y<0$, 
\begin{align*}
   \frac{\partial^{i+1} h(x,y)}{\partial x^i\partial y}&=
  -\left(\frac{2}{9}\right)^{1/6}\frac{1}{6}\frac{\partial }{\partial y}\frac{y}{x^{1/6+i}}e^{2y^3/9x}U\left(\frac{7}{6}-i,\frac{4}{3},-\frac{2y^3}{9x}\right)\\
&=-\left(\frac{2}{9}\right)^{1/6}\frac{1}{6}\frac{-3}{x^{1/6+i}}
\biggl(\left(-\frac{1}{3}-\frac{2y^3}{9x}\right)U\left(\frac{7}{6}-i,\frac{4}{3},-\frac{2y^3}{9x}\right)\\
&\hspace{1cm}+\frac{2y^3}{9x}U'\left(\frac{7}{6}-i,\frac{4}{3},-\frac{2y^3}{9x}\right)\\
&=\left(\frac{2}{9}\right)^{1/6}\frac{1}{2}\frac{1}{x^{1/6+i}}e^{2y^3/9x}U\left(\frac{7}{6}-i-1,\frac{1}{3},-\frac{2y^3}{9x}\right)
\biggr),
\end{align*}
where we used (13.4.27) of \cite{AS64} in the final step. 
\end{proof}

Let 
\begin{equation}
\label{alpha}
\alpha(x,y)=\max(|x|^{1/3},|y|).
\end{equation}
\begin{lemma}
  \label{lem:bounds.h}
There exist positive constants $c$ and $C$ such that
\begin{equation}
\label{h-bound}
c\sqrt{\alpha(z)}\leq h(z)\leq C \sqrt{\alpha(z)},\quad z\in\mathbb{R}_+^2.
\end{equation}
Furthermore, the upper bound is valid for all $z$. Function 
$h$ is at least $C^3$ continuous except the half-line $\{z:x=0,y\geq0\}$.  
\\
For the derivatives we have
\begin{align*}
|h_x(x,y)|\le C \alpha(x,y)^{-2.5},&
|h_{xx}(x,y)|\le C \alpha(x,y)^{-5.5},&
|h_{xxx}(x,y)|\le C \alpha(x,y)^{-8.5},\\
|h_y(x,y)|\le C \alpha(x,y)^{-0.5}, &
|h_{yx}(x,y)|\le C \alpha(x,y)^{-3.5},&
|h_{yxx}(x,y)|\le C \alpha(x,y)^{-6.5},\\
|h_{yyx}(x,y)|\le C \alpha(x,y)^{-4.5}&
|h_{yy}(x,y)|\le C \alpha(x,y)^{-1.5},&
|h_{yyy}(x,y)|\le C \alpha(x,y)^{-2.5}.
\end{align*}
 \end{lemma}
Here and throughout the text we denote as $C,c$ some generic
constants.
\begin{proof}
The estimates will follow from Lemma~\ref{derivatives} and the following properties of 
the confluent hypergeometric functions, 
see (13.1.8), (13.5.8) and (13.5.10) of \cite{AS64},
\begin{align}
\label{u.inf}
U(a,b,s)&\sim s^{-a},\quad s\to\infty,\\
\label{u.zero}
U(a,b,s)&\sim \frac{\Gamma(b-1)}{\Gamma(a)}s^{1-b},\quad s\to 0, b\in (1,2),\\
\label{u.zero.2}
U(a,b,s)&\sim \frac{\Gamma(1-b)}{\Gamma(1+a-b)},\quad s\to 0, b\in (0,1).
\end{align}
Asymptotics (\ref{u.inf}), (\ref{u.zero}) and the definition of $h$ immediately imply (\ref{h-bound}).

Function $h$ is obviously infinitely differentiable when $x<0$ or $x>0$. The only problematic zone 
is $x=0,y<0$. Since $h(x,y)=0$ for $x<0,y<0$ all derivatives are equal to $0$. Using the  expressions 
for derivatives found in Lemma~\ref{derivatives} one can immediately see that 
derivatives of $h(x,y)$ go to $0$ as $x\to 0$ for $y<0$ thanks to the exponent $e^{2y^3/9x}$. 

We  continue with  partial derivatives with respect to $x$. 
First, using (\ref{partial.x}) and (\ref{u.inf}) for sufficiently large $A>0$ and $y^3/x>A$, 
$$
\left|\frac{\partial^i h(x,y)}{\partial x^i}\right| \le 
C\frac{y}{x^{1/6+i}}\left(\frac{2y^3}{9x}\right)^{-1/6-i} 
\le C y^{1/2-3i},\quad i\ge 0.
$$
 
For $y<0$, sufficiently large $A$ and $-y^3/9x>A$, 
the same inequality hold since $e^{2y^3/9x}$ is decreasing much faster than 
any power function as $y^3/x\to -\infty$. 
Next, using (\ref{partial.x}) and (\ref{u.zero}) for sufficiently small $\varepsilon>0$ 
and $y: |y|^3/x\le \varepsilon$,
\begin{align*}
\left|\frac{\partial^i h(x,y)}{\partial x^i}\right| \le 
C\frac{y}{x^{1/6+i}}\left(\frac{2y^3}{9x}\right)^{-1/3} 
\le C x^{1/6-i}.
 \end{align*}
Finally, when $y^3/x\in (\varepsilon,A)$, 
\begin{align*}
\left|\frac{\partial^i h(x,y)}{\partial x^i}\right| \le 
C.
 \end{align*}
We can summarise this in one formula 
\begin{align*}
\left|\frac{\partial^i h(x,y)}{\partial x^i}\right| \le 
C \alpha(x,y)^{-1/6-3i},
\end{align*}
where $\alpha(x,y)$ is defined in (\ref{alpha}). This proves the first line of estimates. 

To prove the second line we use a  similar approach. 
First, using (\ref{partial.xy}) and (\ref{u.inf}) for sufficiently large $A>0$ and $y^3/x>A$, 
$$
\left|\frac{\partial^{i+1} h(x,y)}{\partial x^i\partial y}\right| \le 
C\frac{1}{x^{1/6+i}}\left(\frac{2y^3}{9x}\right)^{-1/6-i} 
\le C y^{-1/2-3i},\quad i\ge 0.
$$
 
For $y<0$, sufficiently large $A$ and $-y^3/9x>A$, 
the same inequality hold since $e^{2y^3/9x}$ is decreasing much faster than 
any power function as $y^3/x\to -\infty$. 
Next, using (\ref{partial.x}) and (\ref{u.zero.2}) for sufficiently small $\varepsilon>0$ 
and $y: |y|^3/x\le \varepsilon$,
\begin{align*}
\left|\frac{\partial^{i+1} h(x,y)}{\partial x^i\partial y}\right| \le 
C\frac{1}{x^{1/6+i}} 
= C x^{-1/6-i}.
 \end{align*}
Finally, when $y^3/x\in (\varepsilon,A)$, 
\begin{align*}
\left|\frac{\partial^{i+1} h(x,y)}{\partial x^i\partial y}\right| \le 
C
 \end{align*}
We can summarise this in one formula 
\begin{align*}
\left|\frac{\partial^{i+1} h(x,y)}{\partial x^i\partial y}\right| \le 
C \alpha(x,y)^{-1/2-3i},
\end{align*}
where $\alpha(x,y)$ is defined in (\ref{alpha}). This proves the second  line of estimates. 

To prove the third line we are using fact that $h_{yy}+0.5yh_x=0$. Hence,
$$
|h_{yy}(x,y)|\le C|y||h_x(x,y)|\le C\alpha(x,y)\alpha(x,y)^{-2.5}\le C\alpha(x,y)^{-1.5}.
$$
Next, 
\begin{align*}
  |h_{yyx}(x,y)|\le C|y||h_{xx}(x,y)|\le C\alpha(x,y)
\alpha(x,y)^{-5.5}\le C\alpha(x,y)^{-4.5}.
\end{align*}
Finally, 
\begin{align*}
  |h_{yyy}(x,y)|&\le C|h_x(x,y)|+C|y||h_{xy}(x,y)|\\
&\le 
C\alpha(x,y)^{-2.5}+C\alpha(x,y)\alpha(x,y)^{-3.5}\le C\alpha(x,y)^{-2.5}.
\end{align*}
The proof is complete. 

\end{proof}
Next we require a bound on $f(x,y)$.
\begin{lemma}
  \label{lem:bound.f}
  Let the assumptions of Lemma~\ref{lem:bounds.h} hold and $f$ is
  defined by \eqref{eq:defn.f}.  Let the moment assumptions
  hold. Then,
  $$
    |f(x,y)| \le C \min(1,\alpha(x,y)^{-3/2-\delta}),\quad (x,y)\in \mathbb{R}_+\times\mathbb{R},
  $$
  for some $\delta>0$.
\end{lemma}
\begin{proof}
Let $A$ be a large constant. Then for $(x,y)$ such that 
$\alpha(x,y)\le A$ using the fact that function $h$ is bounded on any compact  we have $
|f(x,y)|\le C. 
$
In the rest of the proof we consider the case  $\alpha(x,y)>A$ where $A$ is sufficiently large. 

According to Lemma~\ref{lem:bounds.h}  function $h$ is at least $C^3$ smooth except the line $(x=0,y\ge 0)$. 
Then, for $t:|t|\le \frac{1}{2}\alpha(x,y)$, by the Taylor formula, 
\begin{align*}
&\Biggl|
h(x+y+t,y+t)-h(x,y)-\Bigl((y+t)h_x(x,y)+th_y(x,y)\\
&\hspace{0.5cm}
+\frac{1}{2}h_{xx}(x,y)(y+t)^2
+h_{xy}(x,y)(y+t)t
+\frac{1}{2}h_{yy}(x,y)t^2
\Bigr)
\Biggr|\\
&\hspace{0.5cm}
\le \sum_{i+j=3}\max_{\theta:|\theta|\le \frac{1}{2}\alpha(x,y)}
\left|\frac{\partial^{i+j}h(x+y+\theta,y+\theta)}{\partial x^i \partial y^j} (y+t)^it^j\right|:=r(x,y,t)
\end{align*}
To ensure that the Taylor formula is applicable we need to check that the set $\{(x+y+t,y+t):|t|\le \frac{1}{2}\alpha(x,y)\}$ 
is sufficiently far away from the half-line  $\{x=0,y>0\}$, where the derivatives of the function $h(x,y)$ are discontinuous. 
First, if $\alpha(x,y)=y$, then $|y+t|\ge \frac{1}{2}|y|$ for any $t:|t|\le \frac{1}{2}y$. Here, we use the fact that 
$\alpha(x,y)>A.$ Then $|y+t|\ge \frac{1}{2}A$ for any $t:|t|\le \frac{1}{2}y$. Second, 
if  $\alpha(x,y)=x^{1/3}$, then, $|x+y+t|\ge |x|-1.5|x|^{1/3}\ge 0.5A $ for sufficiently large $A$. 
This shows that the Taylor formula is valid.

Then, 
\begin{align*}
\left|  \mathbf Eh(x+y+X)-h(x,y)\right|
&\le 
\left| \mathbf E\left[h(x+y+X)-h(x);|X|> \frac{1}{2}\alpha(x,y)\right]\right|
\\
&\hspace{1cm}
+
\left|\mathbf E\left[h(x+y+X)-h(x,y);|X|\le \frac{1}{2}\alpha (x,y)\right]\right|.
\end{align*}
We can estimate the second term in the right-hand side using the Taylor formula above, 
\begin{align*}
&\left|\mathbf E\left[h(x+y+X)-h(x,y);|X|\le \frac{1}{2}\alpha(x,y)\right]\right|\\
&\hspace{0.5cm}\le 
\Biggl|\mathbf E \biggl[
(y+X)h_x(x,y)
+X h_y(x,y)
+\frac{1}{2}h_{xx}(x,y)(y+X)^2\\
&\hspace{1cm}
+h_{xy}(x,y)(y+X)X
+\frac{1}{2}h_{yy}(x,y)X^2
\biggr]\Biggr|\\
&\hspace{1cm}+\Biggl|\mathbf E \biggl[
(y+X)h_x(x,y)
+X h_y(x,y)
+\frac{1}{2}h_{xx}(x,y)(y+X)^2\\
&\hspace{1.5cm}
+h_{xy}(x,y)(y+X)X
+\frac{1}{2}h_{yy}(x,y)X^2;
|X|> \frac{1}{2}\alpha(x,y)
\biggr]\Biggr|\\
&\hspace{1cm}+\mathbf E\left[r(x,y,X);|X|\le \frac{1}{2}\alpha(x,y)\right]\\
&\hspace{0.5cm}:=E_1(x,y)+E_2(x,y)+E_3(x,y).
\end{align*}
First, we can simplify the first term $E_1(x,y)$  using the assumption $\mathbf EX=0,\mathbf EX^2=1$. Then,
\begin{align*}
  E_1(x,y)&=\left|yh_x(x,y)+\frac{1}{2}h_{yy}(x,y)+\frac{1}{2}h_{xx}(x,y)(y^2+1)+h_{xy}(x,y)\right|. 
\end{align*}
Using  the fact that $h(x,y)$ is harmonic (that is $yh_x+\frac{1}{2}h_{yy}=0$) we obtain
\begin{align*}
  E_1(x,y)&=\left|\frac{1}{2}h_{xx}(x,y)(y^2+1)+h_{xy}(x,y)\right|.
\end{align*}
Applying Lemma~\ref{lem:bounds.h} we obtain 
\begin{equation}
  \label{eq:e1}
  E_1(x,y)\le C \alpha(x,y)^{-5.5}\alpha(x,y)^2+C\alpha(x,y)^{-3.5}\le C\alpha(x,y)^{-3.5}.
\end{equation}
Second, using the fact that $\mathbf E|X|^{2+\delta}<\infty$ and the Chebyshev inequality we obtain 
\begin{align*}
  E_2(x,y)&\le C\mathbf E\biggl[
|X|\left(h_x(x,y)+h_x(x,y)\right)
+X^2\left(h_{xx}(x,y)+h_{xy}(x,y)+h_{yy}(x,y)\right)
;|X|>\frac{1}{2}\alpha(x,y)\biggr] \\
&\le C \frac{|h_x(x,y)|+|h_x(x,y)|}{\alpha(x,y)^{1+\delta}}
+C\frac{|h_{xx}(x,y)|+|h_{xy}(x,y)|+|h_{yy}(x,y)|}{\alpha(x,y)^\delta}.
\end{align*}
Applying Lemma~\ref{lem:bounds.h} we obtain  
\begin{equation}
  \label{eq:e2}
  E_2(x,y)\le C\alpha(x,y)^{-3/2-\delta}. 
\end{equation}
Third, applying Lemma~\ref{lem:bounds.h} once again, 
\begin{align*}
  E_3(x,y)&\le C
\max_{\theta:|\theta|\le \frac{1}{2}\alpha(x,y)}
|h_{xxx}(x+y+\theta,y+\theta)|\mathbf E\left[|y+X|^3;|X|\le \frac{1}{2}\alpha(x,y)\right]
\\
&\hspace{0.3cm}+ 
C\max_{\theta:|\theta|\le \frac{1}{2}\alpha(x,y)}
|h_{xxy}(x+y+\theta,y+\theta)|\mathbf E\left[|y+X|^2|X|;|X|\le \frac{1}{2}\alpha(x,y)\right]\\
&\hspace{0.3cm}+ 
C\max_{\theta:|\theta|\le \frac{1}{2}\alpha(x,y)}
|h_{xyy}(x+y+\theta,y+\theta)|\mathbf E\left[|y+X||X|^2;|X|\le \frac{1}{2}\alpha(x,y)\right]\\
&\hspace{0.3cm}+ 
C\max_{\theta:|\theta|\le \frac{1}{2}\alpha(x,y)}
|h_{yyy}(x+y+\theta,y+\theta)|\mathbf E\left[|X|^3|;|X|\le \frac{1}{2}\alpha(x,y)\right]\\
&\le C\alpha(x,y)^{-8.5}\alpha(x,y)\mathbf EX^2
+C\alpha(x,y)^{-8.5}\alpha(x,y)\mathbf EX^2\\
&\hspace{0.3cm}+C\alpha(x,y)^{-3.5}\alpha(x,y)\mathbf EX^2
+C\alpha(x,y)^{-2.5}\alpha(x,y)^{1-\delta}\mathbf E|X|^{2+\delta}\\
&\le C\alpha(x,y)^{-3/2-\delta}.
\end{align*}
We are left to estimate, 
\begin{align*}
&  \left| \mathbf E\left[h(x+y+X,y+X)-h(x,y);|X|> \frac{1}{2}\alpha(x,y)\right]\right|\\
&\hspace{0.5cm}\le 
C\mathbf E[ \alpha(|x+y+X|,|y+X|)^{0.5};|X|> \frac{1}{2}\alpha(x,y)]\\
&\hspace{1cm}+h(x,y)\mathbf P(|X|> \frac{1}{2}\alpha(x,y))\\
&\hspace{0.5cm}\le
C\mathbf E[|X|^{0.5};|X|> \frac{1}{2}\alpha(x,y)]+
C\alpha(x,y)^{0.5}\mathbf P(|X|> \frac{1}{2}\alpha(x,y))\\
&\hspace{0.5cm}\le C\alpha(x,y)^{-3/2-\delta}\mathbf E|X|^{2+\delta},
\end{align*}
where we applied the Chebyshev inequality in the last step 
and Lemma~\ref{lem:bounds.h} in the first step.
This proves the statement of the lemma.
\end{proof}
\begin{lemma}
\label{concentration}
There exists a constant $C$ such that
$$
\sup_{x,y}\mathbf{P}\left(|S_n^{(2)}-x|\leq1, |S_n-y|\leq1\right)\leq\frac{C}{n^2},\ n\geq1
$$
and
$$
\sup_{x}\mathbf{P}\left(|S_n^{(2)}-x|\leq1\right)\leq\frac{C}{n^{3/2}},\ n\geq1.
$$
\end{lemma}
\begin{proof}
In order to prove the first statement one has to apply Theorem 1.2 from Friedland and Sodin \cite{FS07}
with $\vec{a}_k=(k,1)$ and to note that $\alpha$ from that theorem is not smaller than
$cn$ for this special choice of vectors $\vec{a}_k$. The second inequality follows from Theorem 1.1 of the same paper.
\end{proof}

Let 
$$
K_{n,\varepsilon}=\{(x,y): y>0,x \ge n^{3/2-\varepsilon}\}.
$$

\begin{lemma}\label{lem0}
For any sufficiently small $\varepsilon>0$ there exists 
$\gamma>0$ such that for $k\le n$ the following inequalities hold
\begin{eqnarray}
\label{eq00}
 \mathbf E_{z}[h(Z_k);\tau>k]&\le\left(1+\frac{C}{n^{\gamma}}\right)h(z),\quad z\in K_{n,\varepsilon},\\
\label{eq01}
\mathbf E_{z}[h(Z_k);\tau>k] &\ge \left(1-\frac{C}{n^{\gamma}}\right)h(z),\quad z\in K_{n,\varepsilon}.
\end{eqnarray}
\end{lemma}
\begin{proof}
  First, using \eqref{eq:defn.y} we obtain,
  \begin{align*}
    {\mathbf E}_{z} \left[h(Z_k);\tau>k\right]
    & = {\mathbf E}_{z}[Y_k;\tau>k]+\sum_{l=0}^{k-1}\mathbf E_z[f(Z_l);\tau>k]\\
    & ={\mathbf E}_{z}[Y_k] -{\mathbf E}_{z}[Y_k;\tau\le
    k]+\sum_{l=0}^{k-1}\mathbf E_{z}[f(Z_l);\tau>k].
  \end{align*}
  Since $Y_k$ is a martingale, $\mathbf E_z[Y_k]=\mathbf E_z[Y_0]=h(z)$ 
  and  
 $$
 \mathbf E_{z}[Y_k;\tau\le k]={\mathbf E}_{z}[Y_{\tau};\tau\le k].
 $$
  Using the definition of $Y_k$ once again we arrive at
  \begin{align}
    {\mathbf E}_{z} [h(Z_k);\tau>k]
    &= h(z)-\mathbf E_z [h(Z_\tau),\tau\leq k]\nonumber\\\nonumber
    &\hspace{1cm}+\mathbf E_{z} \left[\sum_{l=0}^{\tau-1}f(Z_l);\tau\le k\right]
    +\sum_{l=0}^{k-1}\mathbf E_{z}[f(Z_l);\tau>k]\\
\label{eq42}
&=h(z)+\mathbf E_{z} \left[\sum_{l=0}^{\tau-1}
      f(Z_l);\tau\le k\right] +\sum_{l=0}^{k-1}\mathbf
    E_{z}[f(Z_l);\tau>k],
  \end{align}
  since $h(Z_\tau)=0$. 

For $k\le n$, we can estimate 
\begin{equation}
\mathbf E_{z} \left[\sum_{l=0}^{\tau-1}f(Z_l);\tau\le k\right] 
                    +\sum_{l=0}^{k-1}\mathbf E_{z}[f(Z_l);\tau>k]
\le 
\sum_{l=0}^{n-1}\mathbf E_{z}[|f(Z_l)|].
\label{eq:sum.finite}
\end{equation}
We split the sum in
(\ref{eq:sum.finite}) in three parts,
\begin{align*}
  \sum_{l=0}^{n-1}\mathbf E_{z}[|f(Z_l)|]
  &=f(z)+\mathbf E_{z} \sum_{l=1}^{n-1} \left[ |f(Z_l)|;\max(|S_l^{(2)}|,|S_l|)\le 1\right]\\
  &\hspace{1cm}+\mathbf E_{z} \sum_{l=1}^{n-1} \left[ |f(Z_l)|;|S_l^{(2)}|^{1/3}>|S_l|,\max(|S_l^{(2)}|,|S_l|)>1\right]\\
  &\hspace{1cm}+\mathbf E_{z} \sum_{l=1}^{n-1} \left[ |f(Z_l)|;|S_l^{(2)}|^{1/3}\le |S_l|,\max(|S_l^{(2)}|,|S_l|)> 1\right]\\
  &=:f(z)+\Sigma_1+\Sigma_2+\Sigma_3.
\end{align*}
First, using the fact that $|f(x,y)|\le C$ for $|x|,|y|\le 1$ and Lemma \ref{concentration}, we obtain
\begin{align*}
  \Sigma_1&\le C\sum_{l=1}^\infty \mathbf P_z(|S_l^{(2)}|,|S_l|\le 1)
  \le C \sum_{l=1}^\infty l^{-2}<C.
  \end{align*}

Second, by Lemma~\ref{lem:bound.f},
\begin{align*}
  \Sigma_2&\le  C\sum_{l=1}^{n-1} \mathbf E_{z}\left[|S_l^{(2)}|^{-1/2-\delta/3}\right]\\
  &\le C \sum_{l=1}^{n-1}\sum_{j=1}^{\infty} \mathbf E_{z}\left[|S_l^{(2)}|^{-1/2-\delta/3};j\le |S_l^{(2)}|\le j+1\right] \\
  &\le C \sum_{l=1}^{n-1}\left(\sum_{j=1}^{l^{3/2}} j^{-1/2-\delta/3}\mathbf
  P_z(j\le |S_l^{(2)}|\le j+1)+l^{3/2(-1/2-\delta/3)}\mathbf P_{z}(|S_l^{(2)}|>l^{3/2})\right).
\end{align*}
Now we use the second concentration inequality from Lemma \ref{concentration} to get
an estimate
\begin{align*}
  \mathbf P_{z}(j\le |S_l^{(2)}|\le j+1)\le Cl^{-3/2}.
\end{align*}
Then,
\begin{align*}
  \Sigma_2&\le C \sum_{l=1}^{n-1}\left(l^{-3/2}\sum_{j=1}^{l^{3/2}} j^{-1/2-\delta/3}+l^{-3/4-\delta/2}\right)\\
  &\le C \sum_{l=1}^{n-1} l^{-3/4-\delta/2}\le Cn^{1/4-\delta/2}.
\end{align*}
Similarly,
\begin{align*}
  \Sigma_3&\le  C\sum_{l=1}^{n-1} \mathbf E_{z}\left[|S_l|^{-3/2-\delta};|Y(l)|\ge 1;|S_l^{(2)}|^{1/3}\leq|S_l|\right]\\
  &\le C \sum_{l=1}^{n-1}\sum_{j=1}^{\infty} \mathbf E_{z}\left[|S_l|^{-3/2-\delta};j\le |S_l|\le j+1;|S_l^{(2)}|\le (j+1)^3\right] \\
  &\le C \sum_{l=1}^{n-1}\left(\sum_{j=1}^{l^{1/2}} j^{-3/2-\delta}\mathbf
  P_z(j\le |S_l|\le j+1;|S_l^{(2)}|\le (j+1)^3)+l^{-3/4-\delta/2}\mathbf P_{z}(|S_l|>l^{1/2})\right).
\end{align*}
Using Lemma \ref{concentration} once again, we get
an estimate
\begin{align*}
  \mathbf P_{z}(j\le |S_l|\le j+1;|S_l^{(2)}|\le (j+1)^3)&\le 
C\sum_{i=1}^{(j+1)^3}\mathbf P_{z}(j\le |S_l|\le j+1;|S_l^{(2)}|\in (i,i+1))\\
&\le Cl^{-2}j^3.
\end{align*}
Then,
\begin{align*}
  \Sigma_3&\le C \sum_{l=1}^{n-1}\left(\sum_{j=1}^{l^{1/2}} j^{-3/2-\delta}l^{-2}j^3+l^{-3/4-\delta/2}\right)\\
  &\le C \sum_{l=1}^{n-1} \left(l^{-2}l^{5/4-\delta/2}+l^{-3/4-\delta/2}\right)\le Cn^{1/4-\delta/2}.
\end{align*}
Therefore,
\begin{align*}
  \sum_{l=0}^{n-1}\mathbf
    E_{z}[|f(Z_l)|]\le f(z)+ Cn^{1/4-\delta/2}.
\end{align*}
On the set $K_{n,\varepsilon}$ due to Lemma~\ref{lem:bound.f} and Lemma~\ref{lem:bounds.h},
$$
|f(z)|\le C\max(1,\alpha(z))^{-3/2-\delta}=C\le C\frac{h(z)}{\alpha(z)}
\le C\frac{h(z)}{n^{1/2-\varepsilon/3}} 
$$
We are left to note, see (\ref{h-bound}), that on the set $K_{n,\varepsilon}$
$$
h(z)\ge cn^{1/4-\varepsilon/2}.
$$
Hence, for $z\in K_{n,\varepsilon}$,
\begin{equation}\label{f-ineq}
\sum_{l=1}^{n-1}\mathbf
    E_{z}[|f(Z_l)|]\le Cn^{1/4-\delta/2}
\le C h(z)\frac{n^{1/4-\delta/2}}{n^{1/4-\varepsilon/2}}
\le Ch(z)n^{-\gamma}, 
\end{equation}
where $\gamma$ is positive   for sufficiently small $\varepsilon$. 

\end{proof}

Let 
$$
\nu_n:=\min\{k\ge 0: Z_k\in K_{n,\varepsilon}\}.
$$
\begin{lemma}
\label{lem1}
There exist a constant such that for 
$$
\sup_{z\in\mathbb{R}_+\times\mathbb{R}}\mathbf P_{z}(\nu_n\ge n^{1-\varepsilon},\tau>n^{1-\varepsilon})\le C\exp\{-n^{\varepsilon/4}\}.
$$
\end{lemma}
\begin{proof}
Fix some integer $A>0$ and put $b_n:=A[n^{1-2\varepsilon/3}].$ Define also $R_n:=[n^{1-\varepsilon}/b_n].$
It is clear that
$$
\mathbf{P}_z\left(\nu_n>n^{1-\varepsilon},\tau>n^{1-\varepsilon}\right)\leq
\mathbf{P}_z\left(S^{(2)}_{jb_n}\in[0,n^{3/2-\varepsilon}]\text{ for all } j\leq R_n\right).
$$
It follows from the definition of $S_n^{(2)}$ that
$$
S_{(j+1)b_n}^{(2)}=S_{jb_n}^{(2)}+b_nS_{jb_n}+\tilde{S}_{b_n}^{(2)},
$$
where $\tilde{S}_n^{(2)}$ is an independent copy of $S_n^{(2)}$. From this representation and 
the Markov property we conclude that
\begin{align*}
&\mathbf{P}\left(S^{(2)}_{jb_n}\in[0,n^{3/2-\varepsilon}]\text{ for all } j\leq R_n\right)\\
&\hspace{1cm}\leq
\mathbf{P}\left(S^{(2)}_{jb_n}\in[0,n^{3/2-\varepsilon}]\text{ for all } j\leq R_n-1\right)
Q_{b_n}\left(n^{3/2-\varepsilon}\right)\\
&\hspace{1cm}\leq\ldots\leq
\left(Q_{b_n}\left(n^{3/2-\varepsilon}\right)\right)^{R_n},
\end{align*}
where
$$
Q_n(\lambda):=\sup_{x\in\mathbb{R}}\mathbf{P}\left(S_n^{(2)}\in[x,x+\lambda]\right).
$$
Using the second inequality in Lemma \ref{concentration}, we get
$$
Q_{b_n}\left(n^{3/2-\varepsilon}\right)\leq\frac{C n^{3/2-\varepsilon}}{A^{3/2}(n^{1-2\varepsilon/3})^{3/2}}
=\frac{C}{A^{3/2}}.
$$
Choosing $A$ so large that $\frac{C}{A^{3/2}}\leq\frac{1}{2}$, we obtain
$$
\mathbf{P}\left(\nu_n>n^{1-\varepsilon},\tau_x>n^{1-\varepsilon}\right)\leq
\left(\frac{1}{2}\right)^{R_n}.
$$
Thus, the proof is finished.
\end{proof}
\begin{lemma}
\label{lem-1}
There exist a constant $C$ such that for 
$k\ge n^{1-\varepsilon}$,
$$
\mathbf E_{z}\left[h(Z_n),\nu_n\ge k,\tau>n^{1-\varepsilon}\right]
\le C(1+\alpha(z))^{1/2}\exp\{-n^{\varepsilon/8}\}.
$$
\end{lemma}
\begin{proof}
Using the Cauchy-Schwartz inequality, we obtain
\begin{align*}
&\mathbf E_{z}\left[h(Z_n),\nu_n\ge k\right]\\
&\hspace{1cm}\leq
\left(\mathbf{E}_z\left[h^2(Z_n),\tau>n^{1-\varepsilon}\right]\right)^{1/2}
\left(\mathbf P_{z}(\nu_n\ge k,\tau>n^{1-\varepsilon})\right)^{1/2}\\
&\hspace{1cm}\leq
\left(\mathbf{E}_z\left[h^2(Z_n),\tau>n^{1-\varepsilon}\right]\right)^{1/2}
\left(\mathbf P_{z}(\nu_n\ge n^{1-\varepsilon},\tau>n^{1-\varepsilon})\right)^{1/2}.
\end{align*}
Recalling that $h(z)\leq C(\alpha(z))^{1/2}$ for all $z\in\mathbb{R}_+\times\mathbb{R}$, 
one can easy obtain the inequality
\begin{align*}
\mathbf{E}_z\left[h^2(Z_n),\tau>n^{1-\varepsilon}\right]
&\le C\mathbf{E}_z\left[\alpha(Z_n)\right]
\le \alpha(z)+\mathbf{E}_0\max(M_n^{1/3}n,M_n)\\
&\leq C(1+\alpha(z))n^{3/2},
\end{align*}
where $M_n=\max_{0\le i\le n}S_i$. 
Combining this with Lemma \ref{lem1}, we complete the proof.
\end{proof}

\begin{lemma}\label{lem3}
For any starting point $z$ there exists a limit 
\begin{equation}
\label{v.limit}
V_0(z)=\lim_{n\to \infty}\mathbf E_{z}\left[h(Z_n);\tau>n\right].
\end{equation}
Moreover, this limit is harmonic and strictly positive on $K_+$.
\end{lemma}
\begin{proof}
Fix a large integer $n_0>0$ and 
put, for $m\ge 1$,
$$
n_m=[n_0^{(1-\varepsilon)^{m}}],
$$
where $[z]$ denotes the integer part of $z$. 
Let $n$ be any integer. It should belong to some interval $n\in (n_m,n_{m+1}]$. 
We first split the  expectation into 2 parts,
\begin{align*}
&\mathbf E_{z} [h(Z_n);\tau>n]=E_{1}(z)+E_{2}(z)\\
&\hspace{1cm}:=\mathbf E_z\left[h(Z_n);\tau>n,\nu_n\leq n_m\right]
+\mathbf E_z\left[h(Z_n);\tau>n,\nu_n> n_m\right].
\end{align*}
By Lemma~\ref{lem-1}, since $n_m\ge n^{1-\varepsilon}$,
the second term on the right hand side is bounded by
\begin{align*}
E_{2}(z)\le \mathbf E_z\left[h(Z_n);\tau>n,\nu_n> n_m\right] 
\le C(1+\alpha(z))^{1/2}\exp\{-Cn_m^{\varepsilon/8}\}.
\end{align*}
Then, 
\begin{align*}
 E_{1}(z)&\le\sum_{i=1}^{n_m}\int_{K_{n,\varepsilon}}\mathbf P_{z}\{\nu_n=i,\tau>i, S_i^{(2)}\in da,S_i\in db\}
\mathbf E_{(a,b)}[h(Z_{n-i});\tau>n-i].
\end{align*}
Then, by (\ref{eq00}),
\begin{eqnarray*}
E_{1}(z)\le \left(1+\frac{C}{n^\gamma}\right)
\sum_{i=1}^{n_m}
\int_{K_{n,\varepsilon}}\mathbf P_{z}\{\nu_n=i,\tau>i, S_i^{(2)}\in da,S_i\in db\}h(a,b).
\end{eqnarray*}
Now noting that $K_{n,\varepsilon}\subset K_{n_m,\varepsilon}$, we apply (\ref{eq01}) to obtain
\begin{align*}
E_{1}(z)&\le \frac{\left(1+\frac{C}{n^\gamma}\right)}{\left(1-\frac{C}{n_m^\gamma}\right)}
\sum_{i=1}^{n_m}
\int_{K_{n,\varepsilon}}\mathbf P_{z}\{\nu_n=i,\tau>i, S_i^{(2)}\in da,S_i\in db\}
\mathbf E_{(a,b)}[h(Z_{n_m-i});\tau>n_m-i]\\
&=\frac{\left(1+\frac{C}{n_m^\gamma}\right)}{\left(1-\frac{C}{n_m^\gamma}\right)}
\mathbf E_z[h(Z_{n_m});\tau>n_m,\nu_n\le n_m].
\end{align*}
As a result we have
\begin{align}\label{L3.1}
\mathbf E_{z} [h(Z_n);\tau>n]\leq
\frac{\left(1+\frac{C}{n_m^\gamma}\right)}{\left(1-\frac{C}{n_m^\gamma}\right)}
\mathbf E_{z}[h(Z_{n_m}); \tau>n_m]
+C(1+\alpha(z))^{1/2}\exp\{-Cn_m^{\varepsilon/8}\}.
\end{align}
Iterating this procedure $m$ times, we obtain
\begin{align}\label{L3.2}
\nonumber
&\mathbf E_{z} [h(Z_n);\tau>n]\leq
\prod_{j=0}^{m}\frac{\left(1+\frac{C}{n_m^{\gamma(1-\varepsilon)^j}}\right)}{\left(1-\frac{C}{n_m^{\gamma(1-\varepsilon)^j}}\right)}\times\\
&\hspace{0.1cm}\left(\mathbf E_{z}[h(Z_{n_0}); \tau>n_0]
+C(1+\alpha(z))^{1/2}\sum_{j=0}^{m}\exp\{-Cn_{m-j}^{\varepsilon/8}\}\right).
\end{align}
First of all we immediately obtain that  
\begin{equation}
\label{eq:sup}
\sup_n \mathbf E_{z} [h(Z_n);\tau>n]\le C(z)<\infty. 
\end{equation}

An identical procedure gives a lower bound 
\begin{align}\label{L3.3}
\nonumber
&\mathbf E_{z} [h(Z_n);\tau>n]\geq
\prod_{j=0}^{m}\frac{\left(1-\frac{C}{n_m^{\gamma(1-\varepsilon)^j}}\right)}{\left(1+\frac{C}{n_m^{\gamma(1-\varepsilon)^j}}\right)}\times\\
&\hspace{0.1cm}\left(\mathbf E_{z}[h(Z_{n_0}); \tau>n_0]
-C(1+\alpha(z))^{1/2}\sum_{j=0}^{m}\exp\{-Cn_{m-j}^{\varepsilon/8}\}\right).
\end{align}
For every positive $\delta$ we can choose $n_0$ such that
$$
\left|\prod_{j=0}^{m}\frac{\left(1-\frac{C}{n_m^{\gamma(1-\varepsilon)^j}}\right)}{\left(1+\frac{C}{n_m^{\gamma(1-\varepsilon)^j}}\right)}-1\right|\leq \delta\quad\text{and}\quad \sum_{j=0}^{m}\exp\{-Cn_{m-j}^{\varepsilon/8}\}\leq\delta.
$$
Then, for this value of $n_0$,
$$
\sup_{n>n_0}\mathbf E_{z} [h(Z_n);\tau>n]
\leq (1+\delta)\left(\mathbf E_{z} [h(Z_{n_0});\tau>n_0]+C(1+\alpha(z))^{1/2}\delta\right).
$$  
and 
$$
\inf_{n>n_0}\mathbf E_{z} [h(Z_n);\tau>n]
\geq (1-\delta)\left(\mathbf E_{z} [h(Z_{n_0});\tau>n_0]-C(1+\alpha(z))^{1/2}\delta\right). 
$$  
Consequently,
\begin{align*}
\sup_{n>n_0}\mathbf E_{z} [h(Z_n);\tau>n]
-\inf_{n>n_0}\mathbf E_{z} [h(Z_n);\tau>n]\\
\leq \delta
\mathbf E_{z} [h(Z_{n_0});\tau>n_0]+2C(1+\alpha(z))^{1/2}\delta.
\end{align*}
Taking into account (\ref{eq:sup}) and  that $\delta>0$ can be made arbitrarily small we arrive at the conclusion 
that  the limit in (\ref{v.limit}) exists. 

To prove harmonicity  of $V_0$ note 
that by the Markov property
$$
\mathbf E_z[h(Z_{n+1});\tau>n+1]=\int_{\mathbb{R}_+\times\mathbb{R}}\mathbf P(z+Z\in dz')\mathbf E_{z'}[h(Z_{n+1});\tau>n]
$$
Letting $n$ to infinity we obtain 
$$
V_0(z)=\mathbf E_z[V(Z_1);\tau>1].
$$
The existence of the limit in the right hand side is justified by the dominated convergence theorem and the above estimates 
for $\sup_{n>n_0}\mathbf E_{z} [h(Z_n);\tau>n]$.

Function $V_0$ has the following monotonicity property: if 
$x'\ge x$ and $y'\ge y$ then $V_0(x',y')\ge V_0(x,y)$. Indeed, 
first the function $h$ satisfies this property since $h_x\ge 0 , h_y\ge 0$, see Lemma~\ref{derivatives}. 
Second it clear that the exit time $\tau'\ge \tau$, where $\tau'$ is the exit of time the integrated random walk 
started from $(x',y')$ and $\tau$ is the exit of time the integrated random walk 
started from $(x,y)$. Third, 
\begin{align*}
\widetilde S_n&=y'+X_1+X_2+\ldots+X_n\ge y+X_1+X_2+\ldots+X_n=S_n , \\
\widetilde S_n^{(2)}&=x'+\widetilde S_1+\widetilde S_2+\ldots+\widetilde S_n\ge S_n^{(2)} 
\end{align*}
Therefore, for any $n$,
$$
\mathbf E_{(x',y')}[h(Z_n);\tau>n]\ge \mathbf E_{(x,y)}[h(Z_n);\tau>n]. 
$$
Letting $n$ to infinity we obtain $V_0(x',y')\ge V_0(x,y)$.

It remains to show that $V_0$ is strictly positive on $K_+$.
For every fixed $n_0$ we have $\mathbf E_{(x,y)} [h(Z_{n_0});\tau>n_0]\sim h(x,y)$ as $x,y\to\infty$.
Then, there exist $x_{n_0}, y_{n_0}$ such that
$$
\inf_{n>n_0}\mathbf E_{z} [h(Z_n);\tau>n]
\geq (1-\delta)^2\left(h(z)-C(1+\alpha(z))^{1/2}\delta\right). 
$$  
Taking into account (\ref{h-bound}),we conclude that $V_0(z)$ is positive for all $z$ with $x>x_{n_0}$, $y>y_{n_0}$.
{F}rom every starting point $z\in\mathbb{R}_+^2$ our process visits the set $x>x_{n_0}$, $y>y_{n_0}$ before $\tau$
with positive probability. Then, using the equation $V_0(z)=\mathbf{E}_z[V_0(Z_1),\tau>1]$, we conclude that $V_0(z)>0$.
The same argument shows that $V_0$ is strictly positive on $K_+$.
\end{proof}

\section{Asymptotics for $\tau$}\label{Sec.tau}
The proof of Theorem \ref{main} goes along the same line as the proofs of conditional limit theorems
in our earlier works \cite{DW10,DW11}. For that reason we give a proof of (\ref{T1.2}) only.
(This allows us also to demonstrate  all changes, which are needed for integrated random walks.)
\subsection{Coupling}
We start with some properties of the integrated Brownian motion.
\begin{lemma}
  \label{lem6}
  There exists a finite constant $C$ such that
  \begin{equation}\label{L6.1}
    \mathbf{P}_{(x,y)}(\tau^{bm}>t)\leq C\frac{h(x,y)}{t^{1/4}},\quad x,y>0.
  \end{equation}
  Moreover,
  \begin{equation}\label{L6.2}
    \mathbf{P}_{(x,y)}(\tau^{bm}>t)\sim \varkappa\frac{h(x,y)}{t^{1/4}},\quad \mbox{ as } t\to\infty,
  \end{equation}
  uniformly in $x,y>0$ satisfying $\max(x^{1/6},y^{1/2})\le \theta_tt^{1/4}$ with some
  $\theta_t\to 0$.  
\end{lemma}

\begin{proof}
To prove this lemma we are going to use the scaling property of Brownian motion, which 
immediately gives for any $\lambda>0$, 
\begin{equation}
  \label{eq:scale}
  \mathbf{P}_{(x,y)}(\tau^{bm}>t)
=\mathbf{P}_{(\lambda^3 x,\lambda y)}(\tau^{bm}>t\lambda^2).
\end{equation}
We start with  (\ref{L6.2}). 
Consider first the case $x^{1/3}\ge y$. Putting $\lambda=x^{-1/3}$ in (\ref{eq:scale}) we obtain 
$$
  \mathbf{P}_{(x,y)}(\tau^{bm}>t)
=\mathbf{P}_{(1,yx^{-1/3})}(\tau^{bm}>tx^{-2/3}).
$$
In view of our assumption $tx^{-2/3}\ge \theta_t^{-1/4}\to\infty$.  We  use the continuity of $h(1,t)$  in $t\in[0,1]$ and immediately obtain that the asymptotics 
$$
\mathbf{P}_{(1,yx^{-1/3})}(\tau^{bm}>tx^{-2/3})\sim \varkappa \frac{h(1,yx^{-1/3})}{(tx^{-2/3})^{1/4}} 
$$
hold uniformly in $yx^{-1/3}\in [0,1]$. Then, 
$$
\mathbf{P}_{(x,y)}(\tau^{bm}>t)
\sim \varkappa \frac{h(1,yx^{-1/3})}{(tx^{-2/3})^{1/4}}
=\varkappa \frac{h(x,y)}{t^{1/4}}.
$$
If $x^{1/3}\le y$ then, choosing $\lambda=y^{-1}$ in (\ref{eq:scale}), we obtain 
$$
  \mathbf{P}_{(x,y)}(\tau^{bm}>t)
=\mathbf{P}_{(xy^{-3},1)}(\tau^{bm}>ty^{-2}).
$$
The rest of the proof goes exactly the same way.

To prove (\ref{L6.1}) first notice that the above proof showed  that for 
sufficiently small $\varepsilon>0$ and 
$t^{1/2}>\varepsilon^{-1}\max(x^{1/3},y)$ the bound (\ref{L6.1}) holds. Hence, it is 
sufficient to consider $t^{1/2}\le \varepsilon^{-1}\max(x^{1/3},y)$. 
Using the lower bound in (\ref{h-bound}), we see that
$$
\frac{h(x,y)}{t^{1/4}}\ge \frac{c \max(x^{1/6},y^{1/2})}{(\varepsilon^{-1}\max(x^{1/3},y))^{1/2}}
=c\varepsilon^2>0
$$
for $t^{1/2}\le \varepsilon^{-1}\max(x^{1/3},y)$,.
Therefore, 
$$
\mathbf P_{(x,y)}(\tau^{bm}>t)\le 1 \le \frac{1}{c\varepsilon^2}\frac{h(x,y)}{t^{1/4}}. 
$$
This proves (\ref{L6.1}). 
\end{proof}
We continue with the classical result (see, for example, \cite{M76}) on the quality of the normal approximation.
\begin{lemma}
\label{lem4}
If $\mathbf{E}|X|^{2+\delta}<\infty$ for some $\delta\in(0,1)$, then one can define a Brownian motion $B_t$ 
on the same probability space such that, for any $\gamma$ satisfying $0<\gamma<\frac{\delta}{2(2+\delta)}$,
\begin{equation}\label{L4}
\mathbf{P}\left(\sup_{u\leq n}|S_{[u]}-B_{u}|\geq n^{1/2-\gamma}\right)=o\left(n^{2\gamma+\gamma\delta-\delta/2}\right).
\end{equation}
\end{lemma} 
\begin{lemma}\label{lem5}
  For all sufficiently small $\varepsilon>0$,
  \begin{equation}\label{L6.4}
    \mathbf{P}_z(\tau>n)=\varkappa h(z)n^{-1/4}(1+o(1)),\quad\text{as }n\to\infty
  \end{equation}
  uniformly in $z\in K_{n,\varepsilon}$ such that $\max\{x^{1/3},y\}\le \theta_n
  \sqrt n$ for some $\theta_n\to 0$. Moreover, there exists a constant
  $C$ such that
  \begin{equation}\label{L6.5}
    \mathbf{P}_z(\tau>n)\le C \frac{h(z)}{n^{1/4}},
  \end{equation}
  uniformly in $z\in K_{n,\varepsilon},n\ge 1$. 
\end{lemma}
\begin{proof}
  For every $z=(x,y)\in K_{n,\varepsilon}$ denote
$$
z^\pm=(x\pm n^{3/2-\gamma}, y).
$$
Note also that if we take
$\gamma>\varepsilon$, then $y^{\pm}\in K_{n,\varepsilon'}$ for any $\varepsilon'>\varepsilon$ and sufficiently large $n$.

Define
$$
A_n=\left\{\sup_{u\leq n}|S_{[u]}-B_u|\le n^{1/2-\gamma}\right\},
$$
where $B$ is the Brownian motion constructed in Lemma~\ref{lem4}.
Then, using (\ref{L4}), we obtain
\begin{align}\label{L5.1}
  \nonumber
  \mathbf{P}_z(\tau>n)&=\mathbf{P}_z(\tau>n,A_n)+o\left(n^{-r}\right)\\
  \nonumber
  &\leq \mathbf{P}_{z^+}(\tau^{bm}>n,A_n)+o\left(n^{-r}\right)\\
  &=\mathbf{P}_{z^+}(\tau^{bm}>n)+o\left(n^{-r}\right),
\end{align}
where $r=r(\delta,\varepsilon)=\delta/2-2\gamma-\gamma\delta.$ In the
same way one can get
\begin{equation}
  \label{L5.2}
  \mathbf{P}_{z^-}(\tau^{bm}>n)\leq \mathbf{P}_z(\tau>n)+o\left(n^{-r}\right).
\end{equation}
By Lemma~\ref{lem6},
$$
\mathbf P_{z^\pm}(\tau^{bm}>n)\sim \varkappa h(z^\pm)n^{-1/4}.
$$
It follows from the Taylor formula and Lemma \ref{lem:bounds.h} that
\begin{equation}
  \label{lbounda}
|h(z^\pm)-h(z)|\leq Cn^{3/2-\gamma}\left(\alpha(x\pm n^{3/2-\gamma},y)\right)^{-5/2}
\leq Cn^{1/4+5\varepsilon/6-\gamma}.
\end{equation}
Furthermore, in view of (\ref{h-bound}),
\begin{equation}\label{L5.3a}
h(z)>cn^{1/4-\varepsilon/6},\quad z\in K_{n,\varepsilon}.
\end{equation}
{F}rom this bound and (\ref{lbounda}) we infer that
$$
h(z^\pm)=h(z)(1+o(1)),\quad z\in K_{n,\varepsilon}.
$$ 
Therefore, we have
$$
\mathbf{P}_{z^{\pm}}(\tau^{bm}>n)=\varkappa h(z)n^{-1/4}(1+o(1)).
$$
{F}rom this relation and bounds (\ref{L5.1}) and (\ref{L5.2}) we
obtain
$$
\mathbf{P}_z(\tau>n)=\varkappa
h(z)n^{-1/4}(1+o(1))+o\left(n^{-r}\right).
$$
Using (\ref{L5.3a}), we see that $n^{-r}=o(h(z)n^{-1/4})$ for all
$\varepsilon$ satisfying $r=\delta/2-2\gamma-2\gamma\delta>\varepsilon/6.$
This proves (\ref{L6.4}).  To prove (\ref{L6.5}) it is sufficient to
substitute (\ref{L6.1}) in (\ref{L5.1}).
\end{proof}
\subsection{Asymptotic behaviour of $\tau$}
Applying Lemma~\ref{lem1}, we obtain
\begin{align}\label{T1.10}
  \nonumber
  \mathbf{P}_z(\tau>n)&=\mathbf{P}_z(\tau>n,\nu_n\leq n^{1-\varepsilon})+\mathbf{P}_z(\tau>n,\nu_n> n^{1-\varepsilon})\\
  &=\mathbf{P}_z(\tau>n,\nu_n\leq
  n^{1-\varepsilon})+O\left(e^{-n^{\varepsilon/4}}\right).
\end{align}
Using the strong Markov property, we get for the first term the
following estimates
\begin{align}
  \label{T1.20}
  \nonumber &\int_{K_{n,\varepsilon}}\mathbf{P}_z\left(Z_{\nu_n}\in
    d\tilde{z},\tau>\nu_n,\nu_n\leq
    n^{1-\varepsilon}\right)\mathbf{P}_{\tilde{z}}(\tau>n)
  \leq\mathbf{P}_z(\tau>n,\nu_n\leq n^{1-\varepsilon})\\
  &\hspace{1cm}\leq\int_{K_{n,\varepsilon}}\mathbf{P}_z\left(Z_{\nu_n}\in
    d\tilde{z},\tau>\nu_n,\nu_n\leq
    n^{1-\varepsilon}\right)\mathbf{P}_{\tilde{z}}(\tau>n-n^{1-\varepsilon}).
\end{align}
Applying now Lemma~\ref{lem5}, we obtain
\begin{align}\label{T1.30}
  \nonumber
  &\mathbf{P}_z(\tau>n;\nu_n\leq n^{1-\varepsilon})\\
  \nonumber &=
  \frac{\varkappa+o(1)}{n^{1/4}}\mathbf{E}_z\left[h(Z_{\nu_n});\tau>\nu_n,|M_{\nu_n}|\leq \theta_n\sqrt{n},\nu_n\leq n^{1-\varepsilon}\right]\\
  \nonumber
  &\hspace{0.5cm}+O\left(\frac{1}{n^{1/4}}\mathbf{E}_z\left[|h(Z_{\nu_n});\tau>\nu_n,|M_{\nu_n}|> \theta_n\sqrt{n},\nu_n\leq n^{1-\varepsilon}\right]\right)\\
  \nonumber &=
  \frac{\varkappa+o(1)}{n^{1/4}}\mathbf{E}_z\left[h(Z_{\nu_n});\tau>\nu_n,\nu_n\leq n^{1-\varepsilon}\right]\\
  &\hspace{0.5cm}+O\left(\frac{1}{n^{1/4}}\mathbf{E}_z\left[h(Z_{\nu_n});\tau_x>\nu_n,|M_{\nu_n}|>
      \theta_n\sqrt{n},\nu_n\leq n^{1-\varepsilon}\right]\right),
\end{align}
where $M_k:=\max_{j\leq k}|S_j|$.

We now show that the first expectation converges to $V_0(z)$ and that
the second expectation is negligibly small.
\begin{lemma}\label{lem11}
  Under the assumptions of Theorem~\ref{main},
$$
\lim_{n\to\infty}\mathbf{E}_z\left[h(Z_{\nu_n});\tau>\nu_n,\nu_n\leq n^{1-\varepsilon}\right]=V_0(z).
$$
\end{lemma}
\begin{proof}
Put $T=\tau\wedge n^{1-\varepsilon}$. Since $Y_k$ is a martingale,
$$
\mathbf{E}_z[Y_T]=\mathbf{E}_z[Y_{\nu_n\wedge T}]=\mathbf{E}_z[Y_{\nu_n},\nu_n<T]+\mathbf{E}_z[Y_{T},\nu_n\geq T]
$$
and, consequently,
$$
\mathbf{E}_z[Y_{\nu_n},\nu_n<T]=\mathbf{E}_z[Y_{T},\nu_n<T].
$$
Using the definition of $Y_k$, we have
\begin{align*}
\mathbf{E}_z[h(Z_{\nu_n}),\nu_n<T]=\mathbf{E}_z[h(Z_{T}),\nu_n<T]
-\mathbf{E}_z\left[\sum_{k=\nu_n}^{T-1}f(Z_k),\nu_n<T\right].
\end{align*}
Conditioning on $Z_{\nu_n}$ and applying (\ref{f-ineq}), we obtain
$$
\left|\mathbf{E}_z\left[\sum_{k=\nu_n}^{T-1}f(Z_k),\nu_n<T\right]\right|
\leq \frac{C}{n^{\gamma(1-\varepsilon)}}\mathbf{E}_z[h(Z_{\nu_n}),\nu_n<\tau].
$$
{F}rom this inequality and Lemma \ref{lem-1}, we conclude
\begin{equation}\label{lem11.1}
\mathbf{E}_z[h(Z_{\nu_n}),\nu_n<T]=(1+o(1))\mathbf{E}_z[h(Z_{T}),\nu_n<T]\quad\text{as }n\to\infty.
\end{equation}
Since $h(Z_\tau)=0$, we have  $h(Z_{T})=h(Z_{n^{1-\varepsilon}}){\rm 1}\{\tau>n^{1-\varepsilon}\}$.  
Using Lemma \ref{lem-1} once again, we get
\begin{align*}
\mathbf{E}_z[h(Z_{T}),\nu_n<T]&=\mathbf{E}_z[h(Z_{n^{1-\varepsilon}}),\nu_n<n^{1-\varepsilon},\tau>n^{1-\varepsilon}]\\
&=\mathbf{E}_z[h(Z_{n^{1-\varepsilon}}),\tau>n^{1-\varepsilon}]+O(e^{-n^{\varepsilon/8}}).
\end{align*}
And in view of Lemma \ref{lem3}, 
$$
\lim_{n\to\infty}\mathbf{E}_z[h(Z_{T}),\nu_n<T]=V_0(z).
$$
Combining this relation with (\ref{lem11.1}), we get the desired result.
\end{proof}

\begin{lemma}
\label{lem21}
As $n\to\infty$,
$$
\mathbf{E}_z\left[h(Z_{\nu_n}),\tau>\nu_n,\nu_n\leq n^{1-\varepsilon},|S_{\nu_n}|>\theta_n\sqrt{n}\right]\to0.
$$
\end{lemma}
\begin{proof}
On the event $\nu_n\leq n^{1-\varepsilon}$,
$$
h(Z_{\nu_n})\leq C\alpha(z)+C\max\left\{\left(n^{1-\varepsilon}M_{n^{1-\varepsilon}}\right)^{1/3},M_{n^{1-\varepsilon}}\right\}
$$
and, consequently,
\begin{align}
\label{lem21.1}
\nonumber
&\mathbf{E}_z\left[h(Z_{\nu_n}),\tau>\nu_n,\nu_n\leq n^{1-\varepsilon},|S_{\nu_n}|>\theta_n\sqrt{n}\right]\\
&\hspace{1cm}\leq C\alpha(z)\mathbf{P}\left(M_{n^{1-\varepsilon}}>\theta_n\sqrt{n}\right)+
C\mathbf{E}\left[M_{n^{1-\varepsilon}},M_{n^{1-\varepsilon}}>\theta_n\sqrt{n}\right].
\end{align}
Here we used the fact that if $\theta_n\to0$ sufficiently slow, then
$$
\max\left\{\left(n^{1-\varepsilon}M_{n^{1-\varepsilon}}\right)^{1/3},M_{n^{1-\varepsilon}}\right\}=M_{n^{1-\varepsilon}}
$$
on the set $\left\{M_{n^{1-\varepsilon}}>\theta_n\sqrt{n}\right\}$.

Using now one of the Fuk-Nagaev inequalities, see Corollary 1.11 in \cite{Nag79}, one can easily conclude that both summands
on the right hand side of (\ref{lem21.1}) vanish as $n\to\infty$.
\end{proof}

\begin{lemma}
For any $z\in K_+$ we have 
\begin{equation}
\label{equality}
V_0(z)=V(z)
\end{equation}
\end{lemma}
\begin{proof}
The aim of this lemma is to show that 2 definitions of $V$ coincide. 
The proof follows closely the proof of Lemma~\ref{lem0}. 
Recall equation (\ref{eq42}), 
\begin{align*}
    {\mathbf E}_{z} [h(Z_n);\tau>n]
=h(z)+\mathbf E_{z} \left[\sum_{l=0}^{\tau-1}
      f(Z_l);\tau\le n\right] +\sum_{l=0}^{n-1}\mathbf
    E_{z}[f(Z_l);\tau>n]. 
\end{align*}
It is sufficient to prove that
  \begin{equation}
    \label{eq:sum.finite.22}
    \mathbf E_z \left[\sum_{l=1}^{\tau-1} |f(Z(l))|\right]<\infty.
  \end{equation}
  Indeed, the dominated convergence theorem then implies that
$$
\mathbf E_z \left[\sum_{l=0}^{\tau-1} f(Z(l));\tau_x\le n\right] \to
\mathbf E_z \left[\sum_{l=0}^{\tau-1} f(Z(l))\right]
$$
and
$$
\left|\sum_{l=0}^{n-1}\mathbf E[f(Z(l));\tau>n]\right| \le \mathbf
E_z \left[\sum_{l=0}^{\tau-1} |f(Z(l))|;\tau>n\right]\to 0
$$
since $\tau_x$ is finite a.s. Then,
$$
\mathbf E_z[h(Z_n);\tau>n]\to h(z)+\mathbf E_{z}\sum_{l=0}^{\tau-1}
      f(Z_l)=V(z), 
$$
which proves (\ref{equality}). 

To prove (\ref{eq:sum.finite.22}) we use the fact that we have already proved that 
$$
\mathbf P_z(\tau>n)\sim V_0(z)n^{-1/4}.
$$
We split (\ref{eq:sum.finite.22}) in three parts,
\begin{align*}
  \mathbf E_{z}\sum_{l=0}^{\tau-1}|f(Z_l)|
  &=f(z)+\sum_{l=1}^{\infty}\mathbf E_{z}[|f(Z_l)|;\tau>l]
 \\
  &= f(z)+\sum_{l=1}^{\infty}\mathbf E_{z} \left[ |f(Z_l)|;|S_l^{(2)}|,|S_l|\le 1,\tau>l\right]\\
  &\hspace{1cm}+\sum_{l=1}^{\infty}\mathbf E_{z}  \left[ |f(Z_l)|;|S_l^{(2)}|^{1/3}>|S_l|,\tau>l\right]\\
  &\hspace{1cm}+ \sum_{l=1}^{\infty}\mathbf E_{z} \left[ |f(Z_l)|;|S_l^{(2)}|^{1/3}\le |S_l|,\tau>l\right]\\
  &=:f(z)+\Sigma_1+\Sigma_2+\Sigma_3.
\end{align*}
First, using the fact that $|f(x,y)|\le C$ for $|x|,|y|\le 1$ and Lemma \ref{concentration}, we obtain
\begin{align*}
  \Sigma_1&\le C\sum_{l=1}^\infty \mathbf P_z(|S_l^{(2)}|,|S_l|\le 1)
  \le C \sum_{l=1}^\infty l^{-2}<C.
  \end{align*}
Second, by Lemma~\ref{lem:bound.f},
\begin{align*}
  \Sigma_2&\le  C\sum_{l=1}^{\infty}\mathbf E_{z} \left[|S_l^{(2)}|^{-1/2-\delta/3},\tau>l\right]\\
  & \le C\sum_{l=1}^{\infty}\mathbf P_z(\tau>l/2)\sup_z\mathbf E_{z} \left[|S_{l/2}^{(2)}|^{-1/2-\delta/3}\right]\\
  &\le CV_0(z) \sum_{l=1}^{\infty}l^{-1/4}\sum_{j=1}^{\infty} \sup_z\mathbf E_{z}\left[|S_{l/2}^{(2)}|^{-1/2-\delta/3};j\le |S_{l/2}^{(2)}|\le j+1\right] \\
  &\le CV_0(z)\sum_{l=1}^{\infty}l^{-1/4}\left(\sum_{j=1}^{l^{3/2}} j^{-1/2-\delta/3}\mathbf
  P_z(j\le |S_{l/2}^{(2)}|\le j+1)+l^{3/2(-1/2-\delta/3)}\mathbf P_{z}(|S_{l/2}^{(2)}|>l^{3/2})\right).
\end{align*}
Now we use the second concentration inequality from Lemma \ref{concentration} to get
an estimate
\begin{align*}
  \mathbf P_{z}(j\le |S_l^{(2)}|\le j+1)\le Cl^{-3/2}.
\end{align*}
Then,
\begin{align*}
  \Sigma_2&\le CV_0(z) \sum_{l=1}^{\infty} l^{-1/4}\left(l^{-3/2}\sum_{j=1}^{l^{3/2}} j^{-1/2-\delta/3}+l^{-3/4-\delta/2}\right)\\
  &\le C V_0(z)\sum_{l=1}^{\infty} l^{-1-\delta/2}\le CV_0(z).
\end{align*}
Similarly,
\begin{align*}
  \Sigma_3&\le  C\sum_{l=1}^{\infty} 
\mathbf P_z(\tau>l/2)
\sup_z\mathbf E_{z}\left[|S_{l/2}|^{-3/2-\delta};|Y(l/2)|\ge 1;|S_{l/2}^{(2)}|^{1/3}\leq|S_{l/2}|\right]\\
  &\le CV_0(z) \sum_{l=1}^{\infty}l^{-1/4}\sum_{j=1}^{\infty} \mathbf E_{z}\left[|S_{l/2}|^{-3/2-\delta};j\le |S_{l/2}|\le j+1;|S_{l/2}^{(2)}|\le (j+1)^3\right] \\
  &\le CV_0(z) \sum_{l=1}^{\infty}l^{-1/4}\biggl(\sum_{j=1}^{l^{1/2}} j^{-3/2-\delta}\mathbf
  P_z(j\le |S_{l/2}|\le j+1;|S_{l/2}^{(2)}|\le (j+1)^3)\\
  &\hspace{2cm}  +l^{-3/4-\delta/2}\mathbf P_{z}(|S_{l/2}|>l^{1/2})\biggr).
\end{align*}
Using Lemma \ref{concentration} once again, we get
an estimate
\begin{align*}
  \mathbf P_{z}(j\le |S_{l/2}|\le j+1;|S_{l/2}^{(2)}|\le (j+1)^3)&\le 
C\sum_{i=1}^{(j+1)^3}\mathbf P_{z}(j\le |S_{l/2}|\le j+1;|S_{l/2}^{(2)}|\in (i,i+1))\\
&\le C(l+1)^{-2}j^3.
\end{align*}
Then, 
\begin{align*}
  \Sigma_3&\le CV_0(z) \sum_{l=1}^{\infty}l^{-1/4}
\left(\sum_{j=1}^{l^{1/2}} j^{-3/2-\delta}l^{-2}j^3+l^{-3/4-\delta/2}\right)\\
  &\le CV_0(z) \sum_{l=1}^{\infty}l^{-1/4} \left(l^{-2}l^{5/4-\delta/2}+l^{-3/4-\delta/2}\right)\\
&\le CV_0(z) \sum_{l=1}^{\infty}l^{-1-\delta/2 }\le CV_0(z).
\end{align*}
This proves that the sum (\ref{eq:sum.finite.22}) is finite. 

\end{proof}

\end{document}